\documentclass{article} 
\usepackage[utf8]{inputenc}
\usepackage[english]{babel}
\usepackage{amsmath}
\usepackage{amssymb}
\usepackage{amsthm}
\usepackage{enumerate}
\usepackage{color,graphicx}
\usepackage{lastpage209}

\numberwithin{equation}{section}

\newtheorem{theorem}{Theorem}[section] 
\newtheorem{lemma}[theorem]{Lemma}
\newtheorem{corollary}[theorem]{Corollary}
\newtheorem{proposition}[theorem]{Proposition}
\theoremstyle{definition}
\newtheorem{definition}[theorem]{Definition}
\newtheorem{remark}[theorem]{Remark}

\newtheorem{example}[theorem]{Example}

\numberwithin{equation}{section}

\def\F{{\mathbb F}}

\def\A{\mathcal A}

\def\L{\mathcal L}

\def\SS{\mathcal S}

\def\dim{{\rm dim}\,}

\title{Values of the length function \\ for  nonassociative algebras \footnote{The work was financially supported by the grant RSF 21-11-00283.}}

\author{A. Guterman, D. Kudryavtsev}
\date{}

\newcommand{\Addresses}{
	\bigskip
	\footnotesize
	
	\textsc{Alexander Guterman, }{Faculty of Algebra, Department of
		Mechanics and Mathematics, Lomonosov Moscow State University, Moscow
		119991, Russia; Moscow Center for Fundamental and Applied Mathematics,  Moscow,   119991, Russia; Moscow Center for Continuous Mathematical Education,  Moscow, 119002, Russia
	}\par\nopagebreak
	\textit{E-mail address}:   \texttt{guterman@list.ru}
	
	\medskip

	\textsc{Dmitry Kudryavtsev, }{School of Mathematics, University of Manchester, Manchester M13 9PL, UK; Moscow Center for Fundamental and Applied Mathematics,  Moscow, 119002, Russia}\par\nopagebreak
	\textit{E-mail address}:   \texttt{dmitry.kudryavtsev@postgrad.manchester.ac.uk}
}

\begin{document}
\maketitle

\Addresses

\begin{abstract}
We study realizable values of the length function for unital possibly nonassociative algebras of a given dimension. To do this we apply  the method of characteristic sequences and establish  sufficient conditions of realisability for a given value of length. The proposed conditions are based on binary decompositions of the value and algebraic constructions that allow to modify length function of an algebra.  Additionally we provide a classification of unital algebras of maximal possible length in terms of their basis.

\end{abstract}


\section{Introduction}

Let $\F$ be an arbitrary field. In this paper $\A$ denotes a finite dimensional unital not necessarily associative $\F$-algebra with the multiplication $(\cdot)$ usually denoted by the concatenation. Let $\SS=\{a_1,\ldots,a_k\}$ be a finite generating set of $\A$. Any product of a finite number of elements from $\SS$ is a {\em word} in $\SS$. The {\em length} of the word $w$, denoted $l(w)$, equals to the  number of letters in the corresponding product.
We consider $1$ as a word in $\SS$ with the {\em length $0$}.

The set of all words  in $\SS$ of lengths less than or equal to $i$ is denoted by $\SS^i$, here $i\ge 0$.

Note that similarly to the associative case, $m<n$ implies that $\SS^m \subseteq \SS^n$.

The set $\L_i(\SS) = \langle \SS^i \rangle$  is the linear span  of  the set  $\SS^i$
(the set of all finite linear combinations with coefficients belonging to
$\mathbb{F}$). We write $\L_i$ instead of $\L_i(\SS)$ if $\SS$ is clear from the context. It should be noted that for unital algebras $\L_0(\SS)=\langle
1 \rangle=\mathbb{F}$ for any $\SS$, and for non-unital algebras $\L_0 = \emptyset$.
We denote  $\L(\SS) =\bigcup\limits_{i=0}^\infty \L_i(\SS)$. 

Since the set $\SS$ is generating for $\A$, the equality $\A=\L(\SS)$ holds.

\begin{definition}\label{sys_len} 
	The {\em length of  a generating set} $\SS$ of a finite-dimensional algebra $\A$ is defined as follows: $l(\SS)=\min\{k\in \mathbb{Z}_+:\L_k(\SS)=\A\}.$ 
\end{definition}

\begin{definition}\label{alg_len} The
	{\em length of an algebra $\A$} is   $l(\A)=\max \{l(\SS): \L(\SS)=\A\}$. 
\end{definition}

The problem of the associative algebra length computation was first discussed in \cite{SpeR59, SpeR60} for the algebra of $3\times3$ matrices in the context of the mechanics of isotropic continua.

It is straightforward to see that the length of a unital associative algebra is strictly less than its dimension, and this bound is sharp. Namely, one-generated unital associative algebra of the dimension $d$ has length $d-1$. The first non-trivial result in this direction is going back to Paz \cite{Paz84}. More results on the length of abstract associative algebras can be found, for example, in~\cite{LafMSh,Mar09,Pap97} and their bibliography. 

In general most of the known results on the length function are just bounds that are not sharp. 
Even the sharp upper bound for the length of the matrix algebra is not known through  these bounds are important in the number of applications, see~\cite{Mar09} and references therein. In particular, the knowledge of the universal upper bound for the length provides the maximal size of products in generators that we need to consider to verify a certain property, for example, unitary similarity, \cite{AlIk}. Moreover, a great deal of work has been done investigating the related notion of length for given generating sets of matrices as in Definition \ref{sys_len}, see~\cite{GutLMSh,Long1,Long2} and references therein. 

The results on the lengths of nonassociative algebras were obtained in the works   \cite{GutK19,GutK21,GutK18}. The common tool to work without associativity  was the method of characteristic sequences which are defined as follows.

\begin{definition} \label{CharSeqUn} \cite[Definition 3.1]{GutK18}
	Consider a unital $\F$-algebra $\A$  of  dimension $\dim \A = n$,
	and its generating set $\SS$. By the {\em characteristic sequence} of $\SS$
	in $\A$ we understand a monotonically non-decreasing sequence of  non-negative integers $(m_0,m_1,\ldots, m_N)$, constructed by the following rules:
	\begin{enumerate}
		\item $m_0 = 0$.
		\item Denoting $s_1=\dim \L_1(\SS)-1$, we define $m_1=\ldots =m_{s_1}=1$.
		\item Let for some $r>0$, $k>1$ the elements $m_1,\ldots,m_r$ be already defined and the sets $ \L_1(\SS) ,\ldots , \L_{k-1}(\SS) $ are considered. Then we inductively continue the process in the following way. Denote $s_k=\dim \L_k(\SS) - \dim \L_{k-1}(\SS)$ and  define $m_{r+1} =\ldots =m_{r+s_k}=k$.  
	\end{enumerate}
\end{definition}

Such sequences  are directly connected to the dimension of algebra and the length of the respective generating set. Combinatorial properties of these sequences, in turn, allow us to establish various facts about lengths. In particular, using characteristic sequences it is possible to  provide a strict upper bound on length of a general nonassociative unital algebra.

\begin{theorem}[{\cite[Theorem 2.7]{GutK18}}]
	Let $\A$ be a unital $\F$-algebra,  $\dim \A = n \ge 2$. Then $l( \A)\le 2^{n-2}$.
\end{theorem}

The above bound is strict, see \cite[Example 2.8]{GutK18}.

To compare associative and nonassociative cases let us observe that  the maximal possible length of an associative algebra of dimension $n$ is considerably smaller, namely it is equal to $(n-1)$. Moreover, in the present paper we show that for all $n$ and any positive integer $l\le n-1$ there are associative algebras of dimension $n$ and length $l$, see Example~\ref{Ex_ass}.

Therefore the following question appears naturally. 

What values of length between $1$ and $2^{n-2}$ can be obtained as lengths of general nonassociative algebras of dimension $n$? 


As it was shown in \cite{GutK19}, the nonassociative situation is quite different and a gap in the  set of values of length was discovered. In particular, already for the value $l=2^{n-2}-1$ there are no algebras of length~$l$.  

In the present paper we investigate the set of realizable values for the length function on nonassociative algebras.
Our main purpose   is to provide an    answer to the above question. As another application of the proposed method we characterize all algebras of the maximal length in terms of their basis.

Instead of the previously known combinatorial criteria based on characteristic sequences, which was nontrivial to check, see  \cite[Theorem 3.19]{GutK19},  we provide an efficient numerical  sufficient condition for feasibility of length values. Our condition is stated in terms of the binary decomposition of an integer under consideration and is easy to check. 

In particular, we show that for algebras of dimension 2, 3, 4 all values of length are realizable, and for any dimension $n\ge 5$ there are gaps of non-realizable values.  If   the value is bigger than a half of the maximal length value,  then  all the gaps are   completely determined. Namely, we prove that   only realizable values for $k>2^{n-3}$  are of the form $2^{n-3}+2^q$, and moreover, all integers of the type $2^h$ and $2^{p}+2^q$ with $h\le n-2$, $q\le p \le n-3$, are  realizable. If $k>\lceil\frac{n}{2} \rceil$ then all values in the interval $[2^{n-k-1},2^{n-k}-1]$ are realizable.   As a corollary a lower bound for the first non-realizable value is obtained.  Namely, it is proved that all values between 1 and $ 2 ^{\lceil \frac{n}{2} \rceil} $ are realizable for $n\ge 6$. We also compute the total number of realizable values in each interval of the form  $[2^{n-k-1},2^{n-k}-1]$. 

Sufficient condition is provided for a given integer to be realizable as a value of length. Namely, if $l<2^{n-k}$ for some  $k$ such that the binary decomposition of $l$ contains less than or equal to $k$ entries 1, then $l$ is realizable. An example is given that this condition is not necessary. A lower bound for the number of realizable values is obtained and several new gaps in the set of values of the length function are found. 

Additionally,  we expand upon connection of characteristic sequence to respective generating set. This is achieved by constructing a special basis $W = \{e_0,\ldots,e_{n-1}\}$ of words corresponding to the characteristic sequence $M = \{m_0,\ldots,m_{n-1} \}$ such that the multiplication low of the elements of $W$ is defined by the addition in~$M$.
As an application of this technique we characterize nonassociative algebras of maximal length in terms of their basis.

Our paper is organized as follows. In Section 2 we provide previously established results and notions, relevant for further proofs. In Section 3 we introduce basic numerical results for possible length values examining  proto-characteristic sequences. Section 4 covers the application of binary decomposition of a given value in checking its feasibility as a length of an algebra. Throughout Sections 5 and 6 we   classify  algebras of maximal length in terms of their basis.

\section{Basic results and notions}

To formulate our results let us start by introducing basic notions and tools in the nonassociative situation.

\begin{definition}\label{PrCharSeq}\cite[Definition 3.1]{GutK19}
	We call the sequence $M=(m_0, \ldots, m_{n-1})$ {\em proto-characteristic} if it satisfies the following
	properties:
	\begin{enumerate}
		\item $m_0=0$.
		\item  There exists $k_1,\ 1\le  k_1 \le n-1,$ such that $m_1=\ldots=m_{k_1}=1$ and if $k_1 < n-1$ it holds that $m_{k_1 +1}>1$.
		\item The sequence $(m_0, \ldots, m_{n-1})$ is non-decreasing. 
		\item  If $k_1 < n-1$ then there exist two functions, $t_1(k)$ and $t_2(k)$,  mapping the set $\{k_1 +1,\ldots,n-1\}$ to $\{1,\ldots,n-1\}$, and satisfying two additional properties:
		\begin{itemize} \item[a.] For $k$ such that $k_1 <k <n$ the equality  $m_{t_1(k)}+m_{t_2(k)}=m_k$ and inequalities $t_1(k), t_2(k) <k$ hold. \item[b.] For all $h_1,h_2$ such that $k_1<h_1<h_2<n$   at least one of the following two inequalities holds: $t_1(h_1)<t_1(h_2)$ or $t_2(h_1)<t_2(h_2)$. \end{itemize}
	\end{enumerate}
\end{definition}

\begin{proposition} \label{pr_full}\cite[Proposition 3.2]{GutK19}
	Characteristic sequence of any generating set of any finite-dimensional unital algebra is proto-characteristic.
\end{proposition}

\begin{theorem}\label{th_ex}\cite[Theorem 3.19]{GutK19}
	Let    $M=(m_0, \ldots, m_{n-1})$, where $n \ge 2$, be a proto-characteristic sequence.
	
	Then there  exists a  unital algebra $\A$ over a field $\F$ with a generating sequence $\SS$, such that  the following conditions hold:
	
	\begin{enumerate}
		\item $\dim \A =n$.
		\item Characteristic sequence of $\SS$ coincides with $(m_0, \ldots, m_{n-1})$.
		\item $l(\A)=m_{n-1}$.
	\end{enumerate}
\end{theorem}

\begin{definition}\label{Def_Irr}
	A word $w$  from a generating set $\SS$ of an algebra $\A$ is   {\em irreducible}, if for each integer $m,\ 0 \leq m<l(w),$ it holds  that $w \notin L_m(\SS)$.
\end{definition}

\begin{lemma}\label{lem_1} \cite[Lemma 2.14]{GutK18}
	An irreducible word of the length greater than 1 is a product of two irreducible words of non-zero lengths.
\end{lemma}

\begin{lemma}\cite[Lemma 3.4]{GutK18}\label{lem_chseq}

 Let $\A$ be an $\F$-algebra, $dim \A =n >2$, and $\SS$ be a generating set for $\A$. Then

\begin{enumerate}

\item Positive integer $k$ appears in the characteristic sequence as many times as many there are linearly independent irreducible words of length $k$.
\item For any term $m_h$ of the characteristic sequence of $\SS$ there is an irreducible word in $\SS$ of length $m_h$.
\item If there is an irreducible word in $\SS$ of length $k$, then $k$ is included into the characteristic sequence of $\SS$.
 
\end{enumerate}

\end{lemma}

\begin{lemma}\cite[Lemma 3.5]{GutK18} \label{N=n-1}
Let $\A$ be an $\F$-algebra, $\dim \A =n >2$, and $\SS$ be a generating set for A. Then the characteristic sequence of $\SS$ contains $n$ terms, i.e., $N=n -1$. Moreover, $m_N=l(\SS)$.
\end{lemma}

\begin{lemma}\label{lem_redux}\cite[Lemma 2.11]{GutK18}
If $\A$ is an algebra and $\SS_0$ and $\SS_1$ are its finite subsets such that $\L_1(\SS_0) =\L_1(\SS_1)$, then $\L_k(\SS_0) =\L_k(\SS_1)$ for every positive integer~$k$.
\end{lemma}

\begin{proposition}\label{prop_max}\cite[Theorem 2.7]{GutK18}
Let $\A$ be an $\F$-algebra, $\dim \A =n > 2$. Then $l(\A) \le 2^{n-2}$.
\end{proposition}

\begin{proposition}\label{prop_add}\cite[Proposition 3.7]{GutK18}
Let $\A$ be an $\F$-algebra, $\dim \A =n >2$. Assume that $\SS$ is a generating set for $\A$ and $(m_0, m_1, \ldots, m_{n-1})$ is the characteristic sequence of $\SS$. Then for each $h$ satisfying $m_h\ge 2$ it holds that there are indices $0 <t_1\le t_2<h$ such that $m_h=m_{t_1}+m_{t_2}$.
\end{proposition}

\begin{proposition}\label{prop_bound}\cite[Theorem 3.8]{GutK18}
 Let $\A$ be an $\F$-algebra,   $\dim \A =n$, $n >2$. Let also $\SS$ be a generating set for $\A$,  $(m_0, m_1, \ldots, m_{n-1})$ be the characteristic sequence of $\SS$. Then for each positive integer $h \le n -1$ it holds that $m_h\le 2^{h-1}$.

\end{proposition}

\section{General observations}

As it was already mentioned in the introduction, the length of a unital associative algebra is strictly less than its dimension and the corresponding bound is sharp. Moreover, we note that there are unital associative algebras of a given dimension $n$ and any length between $1$ and~$(n-1)$ as the following example shows.

\begin{example} \label{Ex_ass}
	Let us consider an arbitrary field $\F$ and an associative $\F$-algebra $\A_l$  with the  basis \{$e_0=1 , e_1, \ldots, e_{n-1}$\} and the multiplication determined by the following rule:

	$$e_p e_q = \begin{cases} e_p, & q=0; \\ e_q, & p=0; \\ e_{p+q}, & p \ge 1, q \ge 1, p+q \le l; \\  0, & \mathrm{otherwise.} \end{cases}$$
	
	It is straightforward to check that  $\A_l$ is associative and  $\dim(\A_l)=n$.
	
	Consider the generating set $S_0=\{e_1,e_{l+1},e_{l+2},\ldots,e_{n-1}\} $ of $\A_l$. Since  $e_l = e_1^l$ is an irreducible word of the length $l$, we have $l(\A_l) \ge l(S_0) = l$. 
	
	Let $\SS$ be an arbitrary generating set of $\A_l$ satisfying $l(\SS)= l(\A_l)$. We  consider the set $\SS'$, obtained from $\SS$ by nullifying coefficients at $e_0$ in the expansions of elements of $\SS$ via the basis $\{e_1,\ldots,e_{n-1}\}$. The equality $\L_1(\SS) =\L_1(\SS')$ holds. Hence by Lemma \ref{lem_redux} it follows that $\L_k(\SS) =\L_k(\SS')$ for every positive integer $k$. In particular this allows to conclude that $\SS'$ is a generating set and $l(\SS) = l(\SS')$. However, we have $l(\SS') \le l$ as a product of any $l+1$ elements of $\SS'$ is zero due to the fact that a product of any $l+1$ non-unit basis elements $e_{i_1},\ldots,e_{i_{l+1}}$ is zero. Thus  $l(\A_l) \le l$, which implies $l(\A_l) = l$.
\end{example}

Below $\A$ is again nonassociative algebra and we discuss some general properties of the values of its length. 

\begin{proposition}\label{pr_mon}
	Let $\A$ be a unital algebra over a field $\F$ such that $\dim \A = n$ and $l(\A) = l$. Then there exists a unital algebra $\A'$ such  that $\dim \A' = n+1$ and $l(\A')=l$.
\end{proposition}
\begin{proof}
	Let $\SS$ be a generating set of $\A$ with maximal length, $M=(m_0, \ldots, m_{n-1})$ be its generating sequence. By Lemma \ref{N=n-1} we have $m_{n-1} = l$.

	Note that by Proposition \ref{pr_full} the sequence $M$ is proto-characteristic with certain functions $t_1$ and $t_2$. 

	Consider the sequence $M'=(m'_0,\ldots, m'_n)=(m_0,1,m_1 \ldots, m_{n-1})$. It is also proto-characteristic: Properties 1-3 of the Definition \ref{PrCharSeq} can be straightforwardly checked, and for Property 4 we can define functions $t'_1(k)$ and $t'_2(k)$ as $t_1(k)+1$ and $t_2(k)+1$. To do this we observe that for all $j$ satisfying $m'_j >1$ it holds that $m'_j = m_{j-1}$. 

	By Theorem \ref{th_ex} there exists a unital algebra of dimension $n+1$ and length $m'_n = m_{n-1} = l$.
\end{proof}

\begin{proposition}\label{pr_double}Let $\A$ be a unital algebra over a field $\F$ such that $\dim \A = n$ and $l(\A) = l$. Then there exists a unital algebra $\A'$ such  that $\dim \A' = n+1$ and $l(\A')=2l$.
\end{proposition}
\begin{proof}
	Let $\SS$ be a generating set of $\A$ with maximal length, $M=(m_0, \ldots, m_{n-1})$ be its generating sequence. By Lemma \ref{N=n-1} we have $m_{n-1} = l$. 
	
	Note that by Proposition \ref{pr_full} the sequence $M$  is proto-characteristic with certain functions $t_1$ and $t_2$. 

	Consider the sequence $M'=(m'_0,\ldots, m'_n)=(m_0,m_1 \ldots, m_{n-1}, 2m_{n-1})$.    It is also proto-characteristic: Properties 1-3 of the Definition \ref{PrCharSeq} can be straightforwardly  seen, and for Property 4 we can define functions $t'_1(k)$ and $t'_2(k)$ to be equal to $t_1(k)$ and $t_2(k)$, respectively, for $k <n$ and $t'_1(n)=t'_2(n)=n-1$. 

	By Theorem \ref{th_ex} there exists a unital algebra of dimension $n+1$ and length $m'_n = 2m_{n-1} = 2l$.
\end{proof}

\begin{proposition}\label{pr_plusone} Let $\A$ be a unital algebra over a field $\F$ such that $\dim \A = n$ and $l(\A) = l$. Then there exists a unital algebra $\A'$ such  that $\dim \A' = n+1$ and $l(\A')=l+1$.
\end{proposition}
\begin{proof}
	Let $\SS$ be generating set of $\A$ with maximal length, $M=(m_0, \ldots, m_{n-1})$ be its generating sequence. By Lemma \ref{N=n-1} holds $m_{n-1} = l$. 
	
	Note that by Proposition \ref{pr_full} $M$ satisfies Properties 1-4 of  Definition \ref{PrCharSeq} with certain functions $t_1$ and $t_2$. 

	Consider the sequence $M'=(m'_0,\ldots, m'_n)=(m_0,m_1 \ldots, m_{n-1}, m_{n-1}+1)$.  It is also proto-characteristic: Properties 1-3 of the Definition \ref{PrCharSeq} can be seen easily, and for Property 4 we can define the functions $t'_1(k)$ and $t'_2(k)$ to be equal $t_1(k)$ and $t_2(k)$, respectively, for $k <n$ and $t'_1(n)=1,t'_2(n)=n-1$. 

	By Theorem \ref{th_ex} there exists a unital algebra of dimension $n+1$ and length $m'_n = 2m_{n-1} = l+1$.
\end{proof}

Let us demonstrate that for a given dimension $n$ there indeed exist algebras with the minimal    possible length values. An example of the algebra of the maximal possible value is given in~\cite[Example 2.8]{GutK18}.

\begin{proposition}\label{pr_onex}
	There exists a unital $\F$-algebra of dimension $n\ge 2$ and length~1.
\end{proposition}
\begin{proof}
	Consider the following sequence $(m_0, \ldots, m_{n-1}) =(0,1,\ldots,1)$. It falls under criteria of Theorem \ref{th_ex}, thus the statement follows.
\end{proof}

To construct a concrete example we consider the following algebra.

\begin{example}\label{ex2_1}
	For every natural $n \ge 2$ let us consider an arbitrary field $\F$ and $\F$-algebra $\A$  with the  basis \{$e_0=1 , e_1, \ldots, e_{n-1}$\} and the multiplication determined by the rule: for all $p,q$: $1 \leq p,q \leq n-1$
	\[e_p e_q = 0.\]
	
	The length of this algebra is 1, as for all $u,v \in \A$ it holds that $u v \in \langle u, v \rangle$, i.e. there can be no irreducible words of length 2 or greater.
\end{example}

Observe that there are no gaps in feasible length values for small dimensions, namely $\dim \A\in \{1,2,3,4\}$.   

\begin{proposition}\label{pr_base}
	For algebras of dimension $n=2,3,4$ each value of length between 1 and $2^{n-2}$ is attainable.
\end{proposition}
\begin{proof}
	1. Let $n=2$. Then the statement is straightforward as $2^{n-2}=1$.
	
	2. Let $n=3$. Then the statement follows from Example  \ref{ex2_1} and~\cite[Example 2.8]{GutK18}.
	
	3.  In the case $n=4$ the values that require justification are only 2 and 3, since 1 can be attained  by Example \ref{ex2_1} and 4 can be attained  by~\cite[Example 2.8]{GutK18}.
	
	3.1. $(0,1,1,2)$ falls under Theorem~\ref{th_ex}.
	
	3.2. $(0,1,2,3)$ falls under Theorem~\ref{th_ex} as well.
\end{proof}

It is shown in \cite{GutK18} that  for each $n\ge 5$ the gaps in the set of values of the length function do exist. To investigate their structure we need the following technical statements.

\begin{proposition}\label{pr_powertwo} Let $ n\ge 2 $, $h \le n-2$ be  integers. Then there exists   a unital $\F$-algebra $\A$ with $\dim \A = n$ and $l(\A) =  2^h$.
\end{proposition}
\begin{proof}
	We will prove this statement using an induction on  $n$.

	The base. For $n=2$ the only possible value of $h$ is 0 and the statement holds by Proposition \ref{pr_onex}.

	The step. Assume the statement holds for $n=N-1$, $N \ge 3$, and consider $n = N$.
	
	 The case $h=0$ holds by Proposition \ref{pr_onex}.  For the case $1 \le h \le N-2$ by the induction hypothesis there exists an algebra of dimension $N-1$ and length $2^{h-1}$ as $h-1$ is non-negative and $h-1 \le N-3$. Thus by Proposition \ref{pr_double} there exists an algebra of dimension $N$ and length $2 \cdot 2^{h-1} = 2^{h}$.
\end{proof}

\begin{lemma}\label{lem_doublebound}
Let $\A$ be an $\F$-algebra, $\dim \A =n > 4$. Assume $\SS$ is a generating set for $\A$ and $(m_0, m_1, ..., m_{n-1})$ is the characteristic sequence of $\SS$. Then for all $h$ such that $2 \le h \le n-2$ it holds that  $m_{h+1} \le 2 m_h$.
\end{lemma}

\begin{proof}
If $m_h = 1$ this follows from the fact that $m_{h-1} \ge 1$ .

If $m_h \ge 2$, we have $m_h=m_i+m_j$ for some $i,j \le h-1$ by Proposition \ref{prop_add}. Since the characteristic sequence is monotonically non-decreasing, we have $m_h = m_i + m_j \le m_ {h-1} + m_{h-1} = 2 m_{h-1}$.
\end{proof}

\begin{proposition}\label{pr_2}
	Let $\A$ be a unital $\F$-algebra, $\dim \A = n$, $n>4$, $\SS$ be
	a generating set of $\A$, $(m_0, m_1,\ldots, m_{n-1})$ be the characteristic sequence of $\SS$. Then 
	
	\begin{enumerate}

	\item For each  positive integer  $h \leq n-1$ it holds that either $m_h = 2^{h-1}$
	or $m_h\leq  3 \cdot 2^{h-3}.$ 
	
	\item Moreover, in the case $m_h = 2^{h-1}$ it holds that $m_l=2^{l-1}$ for all $l\le h$.
	
	\end{enumerate}
\end{proposition}

\begin{proof}
Note that $m_1=1$. If 
\begin{equation} m_k = 2  m_{k-1}\label{m_k} \end{equation} for all 
$2 \leq k \leq h$, then $m_h=2^{h-1}$ and Item 2 is straightforward.
	
Otherwise consider the first $k$, $2\le k \le h$ such that $m_{k} \ne 2  m_{k-1}$, i.e. by Lemma \ref{lem_doublebound} $m_{k} < 2  m_{k-1}$.  

1. If $m_k=1$ by Lemma \ref{lem_doublebound} it holds $m_h \le 2 m_{h-1} \le \ldots \le 2^{h-k} m_k = 2^{h-k} \le 3 \cdot 2^{h-3}  < 2^{h-1}$.

2. If $m_k \ge 2$ we have $m_h=m_i+m_j$ for some $1 \le i,j \le k-1$ by Proposition \ref{prop_add}. Note that $i,j$ cannot be equal to $k-1$ simultaneously due to the choice of $k$. Thus $m_k = m_i + m_j \le m_{k-1} + m_{k-2} = 2^{k-2} + 2^{k-3} = 3 \cdot 2^{k-3}.$

By Lemma \ref{lem_doublebound} it holds $m_h \le 2 m_{h-1} \le \ldots \le 2^{h-k} m_k \le  2^{h-k} ( 3 \cdot 2^{k-3}) = 3 \cdot 2^{h-3}  < 2^{h-1}$.
\end{proof}

\section{Binary decomposition}

\subsection{Gap bounds}

In this subsection we further investigate the set of realizable values of the length function. In particular, we use here a new method to determine whether certain value is a feasible length of an algebra of dimension $n$. This method is  based on binary decomposition of~$n$.
	
The theorem below generalizes the results of Proposition~\ref{pr_powertwo}.

\begin{theorem}\label{th_bintwo} Let $n>4$ be integer. Then for each
$p\in \{0,\ldots,n-3\}$ there exists a unital $\F$-algebra  $\A$   with $\dim \A = n $ and  $l(\A)=2^{n-3}+2^{p}$.
\end{theorem}	
\begin{proof}
Consider the sequence $$(0,1,2,\ldots,2^{n-3}, 2^{n-3}+2^{p}).$$ It is easy to see that it is proto-characteristic with functions $t_1(k)=t_2(k)=k-1$ for $k=2,\ldots,n-2$, and $t_1(n-1)=p+1,$ $ t_2(n-1)=n-2$. Thus by Theorem \ref{th_ex} there exists an $\F$-algebra  $\A$   with $\dim \A = n$ and $l(\A)=2^{n-3}+2^{p}$.
\end{proof}

The converse statement is also true and   provides a generalization for Proposition~\ref{pr_2}.

\begin{theorem}\label{th_bintwo1}
Let $\A$ be a unital $\F$-algebra, $\dim \A = n$, $n>4$, $\SS$ be a generating set for $\A$, $(m_0, m_1,\ldots, m_{n-1})$ be the characteristic sequence of $\SS$. Let $m_h > 2^{h-2}$, where $h \in \{4, \ldots, n-1\}$. Then $m_h=2^{h-2}+2^{q}$, $q \in \{0,\ldots,h-2\}$. 
\end{theorem}	
\begin{proof}
We consider  a unital $\F$-algebra $\A$ with $\dim \A = n$, $n>4$, its generating set $\SS$  and the characteristic sequence $(m_0, m_1,\ldots, m_{n-1})$ of $\SS$. 

	Now we will prove that  from $m_h > 2^{h-2}$, where $h \in \{4, \ldots, n-1\}$, it follows that $m_h=2^{h-2}+2^{q}$, $q \in \{0,\ldots,h-2\}$ by induction on $h$.
	
	The base. For $h=4$ according to Proposition \ref{pr_2}, the value $m_4=7$ is impossible. Thus by Proposition \ref{prop_max}, if $m_h> 2^{4-2} =4$, then  $m_h\in \{5 = 2^2 +2^0, 6=2^2+2^1, 8=2^2 +2^2\}$.
	
	The step. If the statement holds for $h=4,\ldots,H-1$, let us prove it for $h=H \ge 5$. Assume the contrary, i.e. that there exists $Q \in\{0,\ldots,H-3\}$  such that $2^{H-2}+2^{Q} < m_H < 2^{H-2}+2^{Q+1}$, or, equivalently $m_H=2^{H-2}+x$, $2^{Q}<x<2^{Q+1}$. As $(m_0,\ldots,m_{n-1})$ is proto-characteristic by Proposition \ref{pr_full}, we can consider respective functions $t_1$ and $t_2$ . By Property 4a of Definition \ref{PrCharSeq}, $m_{H}=m_{t_1(H)}+m_{t_2(H)}$. Let $T=\max \{t_1(H),t_2(H)\}$ and $t=\min \{t_1(H),t_2(H)\}$, so $m_{H}=m_{T}+m_{t}$. Since $(m_0,\ldots,m_{n-1})$ is monotonically non-decreasing, $m_T \ge m_t$, thus $m_T \ge \frac {m_{H}}{2} > 2^{H-3}$. For any $p\in \{0,\ldots,H-2\}$ by  Proposition \ref{prop_bound}  the inequality $m_{p} \le 2^{p-1} \le  2^{H-3}$ is satisfied. Hence $T=H-1$. By assumption, $2^{H-3}+2^{Q-1} <   \frac{m_{H}}{2} \le  m_T$, thus by  induction hypothesis $m_T =m_{H-1} = 2^{H-3} +2^{Q_1}$ for some $Q_1 \in \{0,\ldots,H-3\}$. There are the following two possibilities.
	
	1. $Q_1 = H-3$. Then $m_{H-1}=2^{H-2}$. It follows that $m_{k} =2^{k-1}$ for $p=0,\ldots,H-2$. Since $t\in \{1,\ldots,H-1 \}$, $m_{H}=m_T + m_t = 2^{H-2} +2^{t-1}$, which contradicts the assumption.
		
	2. $Q_1 < H-3$. Then $m_t = m_{H} - m_T = 2^{H-3} - 2^{Q_1} +x \ge 2^{H-4} + x$. By  Proposition \ref{prop_bound} for $p\in \{0,\ldots,H-3\}$ the inequality $m_{p} \le  2^{H-4}$ holds. Hence $t\in \{H-1,H-2\}$. If $t=H-1$, then $m_T = m_t = \frac{m_{H}}{2} = 2^{H-3} +\frac{x}{2}$, which is   impossible if $x$ is odd. Then $x$ is even. By our assumption $\frac{x}{2} \neq 2^q$ for any $q\in \{0,\ldots,H-3\}$, which contradicts the induction hypothesis for $T$.  Therefore, $t=H-2$. Then by the induction hypothesis $m_t =m_{H-2} = 2^{H-3} +2^{Q_2}$ for some $Q_2 \in \{0,\ldots,H-4\}$. Thus, we have the   equalities: \[2^{H-2}+x = m_{H}=m_T+m_t=2^{H-3}+2^{Q_1}+2^{H-4}+2^{Q_2},\] or equivalently \[2^{H-4}+x=2^{Q_1}+2^{Q_2}.\] Again, there are several possibilities.
	
	2.1. $Q_1=H-4$. Then we have $x = 2^{Q_2}$, which contradicts the assumption.
	
	2.2. $Q_2=H-4$. Then we have $x = 2^{Q_1}$, which contradicts the assumption.
	
	2.3. $Q_1,Q_2 <H-4$. Then $2^{Q_1}+2^{Q_2} \le 2 \cdot 2^{H-5} = 2^{H-4} < 2^{H-4} +x$, thus the equality is impossible.
	
	These contradictions conclude the proof.
\end{proof}


Observe that Theorem \ref{th_bintwo} can be generalized even further. 

\begin{proposition}\label{pr_digits}
	Let $n$ and $k$ be integers such that $n >k>1$. Assume that $l$ is a positive integer such that the following two conditions are satisfied
	\begin{enumerate}
		\item $l <2^{n-k}$;
		\item there are no more than $k$ elements 1 in the binary decomposition of~$l$.
	\end{enumerate}
Then there exists an algebra $\A$ of dimension $n$ and length $l$.
\end{proposition}

\begin{proof}
	We will prove this statement using a double induction on $n$ and $k$.
	
	The base for $n$.  If $n =3$ then the only possible values of $k$ and $l$ are $2$ and $1$ respectively and the statement is straightforward.

	The step for $n$. Assume the  statement holds for $n=3,\ldots,N-1$ with $N >3$ and consider $n=N$.

	The base for $k$. If $k=2$ then there are  two possibilities:
	
	1. $l=2^h$, where $0\le h \le N-3$. It is the statement of Proposition \ref{pr_powertwo}.
	
	2. There exist positive integers $h_1,h_2$ such that $h_1<h_2<N-2$ and $l=2^{h_2}+2^{h_1}$. If $h_2=N-3$, then the statement follows directly from Theorem \ref{th_bintwo}. If $h_2<N-3$, then by Theorem \ref{th_bintwo} there exist an algebra of length $l$ and dimension $h_2+3$. It follows by Proposition \ref{pr_mon} that there is an algebra of length $l$ and dimension $N$.
	
	The step for $k$. Assume that we have proven the statement for  $k=2,\ldots,K-1$ where $2<K<N$. Consider $l$ such that $l<2^{N-K}$ and there are no more than $K$ 1s in its binary decomposition. If there are $k_1<K$  1s, the statement follows from the induction hypothesis for $k$, since $l<2^{N-K}<2^{N-k_1}$. If there are exactly $K$ 1s (it should be noted that it means that $K<N-1$ in this case, since if $K=N-1$, then $l$ would be equal to $1$), then we have two possibilities again.
	
	1. $l$ is odd. Note that $l-1$ has $K-1$ elements 1 and $l-1 < l < 2^{N-K}=2^{(N-1)-(K-1)}$. By the induction hypothesis for $n$ there exists an algebra of dimension $N-1$ and length $l-1$. Thus, by Proposition \ref{pr_plusone} there exists an algebra $\A$ of dimension $N$ and length $l$.
	
	2. $l$ is even. Note that $\frac{l}{2}$ has $K$ elements 1 and $ \frac{l}{2}<\frac{2^{N-K}}{2}=2^{(N-1)-K}$. By the induction hypothesis for $n$ there exists an algebra of dimension $N-1$ and length $\frac{l}{2}$. Thus, by Proposition \ref{pr_double} there exists an algebra $\A$ of dimension $N$ and length~$l$.
\end{proof}

The following set of realizable values is also a consequence from Theorem~\ref{th_bintwo}.

\begin{proposition}
	Let $n>4$ be integer. Then for each
	$p\in \{2,\ldots,n-3\}$ and $q\le p$ there exists a unital $\F$-algebra  $\A$   with $\dim \A = n $ and  $l(\A)=2^{p}+2^{q}$.
\end{proposition}

\begin{proof}
	By Theorem \ref{th_bintwo} there exists algebra  $\A$ of dimension $p+3$ and length $2^{p}+2^{q}$. If $p+3 <n$, we can apply Proposition \ref{pr_mon} $n-p -3$ times and get an algebra of dimension $n$ and the same length.
\end{proof}

\begin{proposition}\label{pr_interval}
Let $n,k$ be integers satisfying   $1<k<n$. 

1.  There are at least $1 + \binom{n-k-1}{1} + \ldots + \binom{n-k-1}{\min(k-1, n-k-1)}$   realizable length values  of algebras of dimension $n$ in the interval $[2^{n-k-1}, 2^{n-k} - 1]$. 

2. If  $k > \lceil \frac{n}{2} \rceil$, then all values in the interval $[2^{n-k-1}, 2^{n-k} - 1]$ are realizable.

3. If $k=2$, then there are exactly $1+\binom{n-3}{1} = n-2$ realizable values.

\end{proposition}

\begin{proof}

1. Since $ 2^{n-k} - 1 - 2^{n-k-1} = 2^{n-k-1}-1$, there are exactly $2^{n-k-1}$ values in the interval. By Proposition \ref{pr_digits} for a given $h$ such that $0 \le h \le \min(k-1, n-k-1)$ all of the values below $2^{n-k}$ with $h+1$ elements 1 in binary decomposition are feasible. The number of such values in the interval $[2^{n-k-1}, 2^{n-k} - 1]$ is equal to  $\binom{n-k-1}{h}$ as we have $n-k$ digits with the leading digit equal to 1 and we can place remaining $h-1$ elements 1 into other $n-k-1$ digits arbitrarily. Thus, by summing all binomial coefficients for all possible $h$ we can infer that there are at least $1 + \binom{n-k-1}{1} + \ldots + \binom{n-k-1}{\min(k-1, n-k-1)}$   realizable length values in the interval.

2. If  $k > \lceil \frac{n}{2} \rceil$, then the binomial sum is equal to $2^{n-k-1}$ which means that all of the values in the interval are feasible.

3. If $k=2$, then by  Theorems \ref{th_bintwo} and \ref{th_bintwo1} only $1 + n-3 = n-2$ values in the interval are feasible.

\end{proof}

Observe that the converse statement to Proposition \ref{pr_digits}, in the same sense as Theorem \ref{th_bintwo1} is the 'converse' to Theorem \ref{th_bintwo}, is not true. In particular, there exist algebras of length $x$ which has $k$ 1s in its binary decomposition, but $x \ge 2^{n-k}$, and Proposition \ref{pr_interval} provides only a lower bound on the number of feasible values in specified intervals.

\begin{example}\label{ex_ex}
	 Consider the sequence $(0,1,2,3,5,10,20,23)$ with the functions $t_1(k)$ and $t_2(k)$ defined as follows:
	\begin{center}
		\begin{tabular}{|c|c|c|}
			\hline
			$k$ & $t_1(k)$ & $t_2(k)$ \\
			\hline
			2 & 1 & 1 \\
			\hline
			3 & 1 & 2 \\
			\hline
			4 & 2 & 3 \\
			\hline
			5 & 4 & 4 \\
			\hline
			6 & 5 & 5 \\
			\hline
			7 & 3 & 6 \\			
			\hline
		\end{tabular}
	\end{center}
	
	On the one hand, it is proto-characteristic, and by Theorem \ref{th_ex} there exists an algebra of dimension 8 and length 23. On the other hand, $23=10111_2$. So, there are four elements 1 in the binary decomposition, however $23 > 16= 2^{8-4}$.
\end{example}

We conclude this subsection by the list of all realizable values in the top half of the possible length values. They appear to be rather rare and more or less isolated. Namely, the corollary   below demonstrates that all realizable values of length that are bigger than the half of the maximal value have the form $2^{n-3} +2^p$, where $1\le p\le n-3$. This implies that the all the   intervals $[2^{n-3}+2^p+1,2^{n-3}+2^{p+1}-1]$, $1\le p\le n-3$, are gaps in the the set of values of the length function. 

\begin{corollary} \label{cor:gaps}
	Let $n>4$ be integer. Then for each
	a unital $\F$-algebra  $\A$   with $\dim \A = n $ and  $l(\A) > 2^{n-3}$ we have $l(\A) =2^{n-3}+2^{p}$ for $p\in \{0,\ldots,n-3\}$ .
\end{corollary}

\begin{proof}
	Consider a generating set $\SS$ of $\A$ such that $l(\A) = l(\SS)$ and its characteristic sequence $(m_0, m_1,\ldots, m_{n-1})$. Since $m_{n-1} = l(\A) > 2^{n-3}$, by Theorem \ref{th_bintwo1} we have $m_{n-1}=2^{n-3}+2^{p}$, $p \in \{0,\ldots,n-3\}$.
\end{proof}

\subsection{Bounds for the subsequent values of the length function}

Now we may investigate the upper bound for the  sequence of consequent values of lengths for algebras of a given dimension, which is the same as the  lower bound for the first non-feasible value.

\begin{definition}
Let $n\ge2$ be a natural number. By $B(n)$ we  
understand the maximal integer $l>0$ such that there exist
algebras 
of dimension $n$ and lengths $1,\ldots,l $.
\end{definition}

\begin{proposition}
We have $B(2)=1, B(3)=2, B(4)=4, B(5) = 6$.
\end{proposition}

\begin{proof}
The first three values follow from Proposition \ref{pr_base}, since these are the maximal possible values of length for $n\in \{2,3,4\}$. For $n=5$ note that 7 is not a feasible length value by Corollary \ref{cor:gaps} as $7=2^{5-3} +3 \ne  2^{5-3}+2^p$. Also the values from 1 till 4 are feasible by Proposition \ref{pr_mon} as they are feasible for $n=4$. The value   5 is feasible by Proposition \ref{pr_plusone} as 4 is a feasible length value for dimension 4. The value   6 is feasible by Theorem \ref{th_bintwo} as $6 = 2^{5-3} + 2^1$.
\end{proof}

\begin{proposition}
Let $n \ge 2$ be integer. Then $B(n+2) \ge 2\cdot B (n) + 1$. 
\end{proposition}

\begin{proof}	
  By Propositions \ref{pr_double} and \ref{pr_plusone} there exist algebras of dimension $n+2$ and lengths $2 \cdot 1,\ldots,2 \cdot B(n+1)$ (the first proposition) and $2 \cdot 1 +1,2 \cdot 2 +1,\ldots,2 \cdot B(n)+1$ (both propositions combined).  By Proposition \ref{pr_mon}, $B(h)$ is monotonically non-decreasing, hence $B(n+1) \ge B(n)$ and algebras of dimension $n+2$ with lengths $1,2,\ldots,2 \cdot B(n)+1$ exist, which means $B(n+2) \ge 2\cdot B (n) + 1$.
\end{proof}

\begin{proposition}
Let $n \ge 2$ be integer. Then	$B(n) \ge 2 ^{\lceil \frac{n}{2} \rceil}$ for $n\ge 6$.
\end{proposition}
\begin{proof}
1.  Let $n$ be even, i.e. $n=2h$ for some integer $h\ge 3$. Then by Proposition \ref{pr_interval}, Item 2, every length value $l_0 \in [1, \ldots, 2^h-1]$ is attainable for algebras of dimension $n$, since $l_0$   can be placed into an interval $[2^{2h-k-1}, 2^{2h-k} - 1]$ with $k \ge h$. Thus, $B(2h) \ge 2^h$.
	
2.	 Consider the case   $n$ is odd, i.e. $n=2h-1$ for some integer $h\ge 4$. By Proposition \ref{pr_interval}, Item 2, each integer value $l_0\in [1 , 2^h-1]$ is realizable as a length of an algebra of dimension $n$, since $l_0\in [2^{n-k-1}, 2^{n-k} - 1]$ for some $k \ge h-1$. By  Proposition \ref{pr_digits} all integer values  $l_1 \in [2^{h-1},2^h-2] $ are also realizable since 1)   $l_1< 2^{h}=2^{2h-1-(h-1)}$ and 2) $l_2=2^h -1$ is the only integer with more than $h-1$ elements 1 in its binary decomposition satisfying $l_2<2^{h}$. To show that $ l_2$ is a feasible length value we   use the induction on~$h$.
	
	The base. For $h=4$, $2^4 -1 = 15$. Consider the sequence $(0,1,2,3,5,10,15)$ with $t_1$ and $t_2$ defined as follows:
	\begin{center}
	\begin{tabular}{|c|c|c|}
		\hline
		$k$ & $t_1(k)$ & $t_2(k)$ \\
		\hline
		2 & 1 & 1 \\
		\hline
		3 & 1 & 2 \\
		\hline
		4 & 2 & 3 \\
		\hline
		5 & 4 & 4 \\
		\hline
		6 & 4 & 5 \\
		\hline
	\end{tabular}
	\end{center}
	It is proto-characteristic, and by Theorem \ref{th_ex} there is an algebra of dimension 7 and length 15.
	
	The step. Assume the statement holds for $h=H-1$ with $H>4$. Then for $h=H$ we need to prove that there exists an algebra of length $2^{H}-1$ and dimension $2H - 1$. By the induction hypothesis, there exists an algebra $\A$ with $l(\A)=2^{H-1}-1$ and  $\dim(\A)=2H -3$. Then by Propositions \ref{pr_double}  there exists an algebra $\A_1$ such that $l(\A_1)=(2^{H-1}-1) \cdot 2=2^{H}-2$ and   $\dim(\A_1)=2H-2$. Hence by Proposition \ref{pr_plusone}  there exists an algebra $\A_2$ with $l(\A_2)=2^{H}-2+1 = 2^H - 1$ and  $\dim(\A_2)=2H - 1$. This concludes the induction.
	
	$2^h$ is a feasible value by Proposition \ref{pr_powertwo}. Thus, $B(2h-1) \ge 2^h$.
\end{proof}

\section{Characteristic sequence and basis}

In the next section we characterize the algebras of maximal length. To achieve this goal in the current section we construct a special basis $W = \{e_0,\ldots,e_{n-1}\}$ of an algebra $\A$ corresponding well to a characteristic sequence $M = (m_0,\ldots,m_{n-1})$ of a given generating set $\SS$ of this algebra. 
We will approach this goal gradually. 

\begin{definition}

Consider an $\F$-algebra $A$, its generating set $\SS$ and a characteristic sequence $M=(m_0,\ldots,m_{n-1})$ of $\SS$. Let $k_1$ be such an index that $m_1=\ldots=m_{k_1}=1$ and if $k_1 < n-1$ it holds that $m_{k_1 +1}>1$ and $t_1,t_2$ be the functions mapping $\{k_1 +1,\ldots,n-1\}$ to $\{1,\ldots,n-1\}$, and satisfying the two following properties (i.e. Item 4 of Definition \ref{PrCharSeq}):

\begin{itemize} \item[a.] For $k$ such that $k_1 <k <n$ the equality  $m_{t_1(k)}+m_{t_2(k)}=m_k$ and inequalities $t_1(k), t_2(k) <k$ hold. \item[b.] For all $h_1,h_2$ such that $k_1<h_1<h_2<n$   at least one of the following two inequalities holds: $t_1(h_1)<t_1(h_2)$ or $t_2(h_1)<t_2(h_2)$. \end{itemize}

We say that the basis $W= \{e_0,\ldots,e_{n-1}\}$ of $\A$ {\em corresponds} to the characteristic sequence $M$ equipped with the functions $t_1,t_2$  if the following holds:

\begin{enumerate}
	\item For each $  i \in \{0,\ldots,n-1 \}$ the element $e_i\in W$ is a word in $\SS$ of length~$m_i$.
	\item For every $k$ such that $k_1 <k <n$ the  equality $e_{t_1(k)} \cdot e_{t_2(k)} = e_k$ is true in~$\A$.
	\item $ \L_1(\SS) = \langle \{ e_0, \ldots, e_{k_1} \} \rangle$.
\end{enumerate}

\end{definition}

\begin{remark}
For a characteristic sequence $M$ at least one pair of such functions $t_1,t_2$ exists since $M$ is proto-characteristic by Proposition \ref{pr_full}, which means that Item 4 of Definition \ref{PrCharSeq} applies to it.
\end{remark}

Let us introduce an auxiliary definition.

\begin{definition}\label{def_graded}
	We say that the basis $W= \{e_0,\ldots,e_{n-1}\}$ of the $\F$-algebra $\A$ is {\em graded} by the generating set $\SS$  of $\A$, or just {\em $\SS$-graded} basis, if  there exists a subset $\SS' \subset \SS$, which is also a generating set, and a sequence of sets $W_0 \subset W_1 \subset \ldots \subset W_{l(\SS)}= W$ such that:

\begin{enumerate} 
	
	\item $W_0 =\{1\}$, $W_1 = \SS' \cup  \{1\}$.
	
	\item $W_k$ is a basis of $\L_k(\SS)$ for $k = 1,\ldots, l(\SS)$
	
	\item $W_k \setminus W_{k-1}$ consists of irreducible words of length $k$ in $\SS'$.
	
	\item Each element of $W_k$, which is a word of length at least 2 in $\SS'$, is a product of two non-unit elements from $W_{k}$. 
	
\end{enumerate}
	
\end{definition}

\begin{lemma}\label{lem_subbase}
Let   $\A$ be an $\F$-algebra, $\dim(\A)= n \ge 2$,  $\SS$ be a  generating set of $\A$. Then there exists an $\SS$-graded  basis $W$ of~$\A$.
\end{lemma}

\begin{proof}

Let $\SS' = \{a_1,\ldots,a_p \}$ be a maximal subset of $\SS$ which is linearly independent modulo $\F$. Note that $\F \cap \SS' = \emptyset$, as otherwise $\SS'$ would not be linearly independent modulo $\F$.

Define $W_1$ as $\SS' \cup \{1\}$. We will prove that $W_1$ is a linearly independent set. Consider a linear combination of its elements equal to zero,  $$f 1 + f_1 a_1 + \ldots + f_p a_p =0$$  with coefficients $f,f_1, \ldots, f_p \in \F$. Hence $ f_1 a_1 + \ldots f_p a_p  = -f$. If $f_1 = \ldots = f_p = 0$, then $f = 0$ and the combination is trivial. Otherwise such an equality  contradicts the fact that $\SS'$ is linearly independent modulo $\F$. 

Note that $\L_1(\SS')$ by definition is equal to $\langle W_1 \rangle = \langle  \SS' \cup  \{1\} \rangle =  \langle  \SS \cup  \{1\} \rangle = \L_1(\SS)$.  By Lemma \ref{lem_redux} we have $\L_k(\SS) = \L_k (\SS')$, which implies that $\SS'$ is a generating set. Additionally, for the purposes of Item 2 it is enough to prove that $W_k$ is a basis of $\L_k(\SS')$.

We construct the sequence of $W_k$ satisfying Items 2-4 inductively.

The base for $k=1$ is provided above. Indeed, $W_1$ is a basis of $\L_1(\SS')$ as it is linearly independent and spans $\L_1(\SS')$, $W_1 \setminus W_0 = \SS'$ consists of irreducible words of length 1 in $\SS'$ (as $\F \cap \SS' = \emptyset$, words of length 1 cannot be reduced), and, since in $W_1$ there are no elements which are words of length at least 2 in $\SS'$, Item 4 holds as well. 

The step. Assume $W_1, \ldots, W_{K-1}$ are already constructed. Consider an irreducible word $w$ of length $K \le l(\SS)$. Note that $w = w' \cdot w''$, where $w'$ has length $j$, $1\le j < K$, and $w''$ has length $K-j$ , which is also less than $K$. As $w' \in \L_j(\SS') = \langle W_j \rangle$ and $w'' \in \L_{K-j}(\SS') = \langle W_{K-j} \rangle$, we have $w \in \langle W_j \cdot W_{K-j} \rangle$. From this it follows that all of the irreducible words  of length $K$ belong to $\langle \bigcup\limits_{i=1,\ldots,K-1} W_i W_{K-i} \rangle$. Thus we can expand $W_{K-1}$  with irreducible words belonging to $W_i W_{K-i}$  for some indices $i$ to obtain the basis of $\L_K(\SS')$. We name the resulting set $W_K$, as it satisfies Items 2-4: 

i) It is indeed a basis by construction.

ii) Only irreducible words of length $K$ in $\SS'$ were added, as any word of length less than $K$ belongs to $\langle W_{K-1} \rangle$ already.

iii) Each added element is a product of two elements from $W_{K-1}$, which, by construction, is a subset of $W_K$.

\end{proof}

\begin{lemma}\label{lem_meekbase}
Let   $\A$ be an $\F$-algebra, $\dim(\A)= n \ge 2$,  $\SS$ be a  generating set of $\A$. Let  $M=(m_0,\ldots,m_{n-1})$  be the characteristic sequence of $\SS$  and $W$ be an $\SS$-graded basis of $\A$ with associated set $\SS' \subset \SS$ and sequence of sets $W_0 \subset W_1 \subset \ldots \subset W_{l(\SS)}= W$. Then there exists a numbering   $e_i$, $i=0,\ldots,n-1$, of the elements from $W$ such that 

\begin{enumerate}

\item $e_0 =1$

\item For all $r$ such that $1 \le r \le l(\SS)$ it holds that $W_r = \{e_0,\ldots, e_{s_1 + \ldots +s_r} \}$, where $s_r = \dim \L_r(\SS) - \dim \L_{r-1}(\SS)$. Additionally, the length of $e_j$ as a word in $\SS'$ is equal to $m_j$ for $j=0,\ldots,s_1 +\ldots +s_r$.

\item Under this numbering for all $r$  such that $1 \le r \le l(\SS)$ there exist functions $t_1^{(r)}$ and $t_2^{(r)}$ from the set $\{s_1 +1,\ldots,s_1+\ldots+s_r\}$ to $\{1,\ldots,s_1+\ldots+s_r\}$ for which  the following two properties are satisfied.

\begin{itemize}

\item[a.] For $k$ such that $s_1 <k \le s_1 +\ldots +s_r$ we have  $e_{t_1^{(r)}(k)} e_{t_2^{(r)}(k)} = e_k$. 

\item[b.] For all $h_1,h_2$ such that $s_1<h_1<h_2 \le s_1 + \ldots + s_r$  at least one of the inequalities $t_1^{(r)}(h_1)<t_1^{(r)}(h_2)$ or $t_2^{(r)}(h_1)<t_2^{(r)}(h_2)$ is satisfied. 

\end{itemize}

\end{enumerate}

\end{lemma}

\begin{proof}
We will construct this numbering and respective functions using induction on $r$.

The base.  For $r=1$ set $e_0 = 1$ and $e_1,\ldots,e_{s_1}$ to be elements of $\SS'$ in an arbitrary order. This guarantees Item 1 and Items 2 and 3 for $r=1$ with $t_1^{(1)}$ and $t_2^{(1)}$ having empty domain.

The step. Assume that Items 2 and 3 hold for $r=1,\ldots,R-1$, with $2 \le R \le l(\SS)$. For $r=R$ there are two possibilities.

1. If $s_R = 0$, then $$W_R = W_{R-1} =  \{e_0,\ldots, e_{s_1 + \ldots +s_{R-1}} \} = \{e_0,\ldots, e_{s_1 + \ldots +s_{R-1}+s_R}. \}$$ The functions $t_1^{(R)}$ and $t_2^{(R)}$ can be defined as $t_1^{(R-1)}$ and $t_2^{(R-1)}$ with Item 3 holding by the induction hypothesis.

2. If $s_R >0$, then consider   all the elements of $W$ which are the words of length $R$ in $\SS'$ (note that all of them belong to $W_R$ by Item 3 of Lemma \ref{lem_subbase}). Since by Item 4 of Lemma \ref{lem_subbase} they are equal to products of two shorter, already numbered elements of $W$, we can establish a correspondence between them and pairs of indices: a word $w$ of length $R$ would correspond to $(j_1,j_2)$, where $e_{j_1} e_{j_2} = w$. Since all the words in $W$ are distinct, all the pairs in the resulting correspondence are also distinct. 

By sorting these pairs in lexicographic order we can continue the numbering of elements of $W$. We name the one with  the lexicographic  minimal corresponding pair $e_{s_1 + \ldots +s_{R-1}+1}$, the next $e_{s_1 + \ldots +s_{R-1}+2}$ and continue in this way. Item 2 will hold for $r=R$ as there are exactly $s_R$ words of length $R$ in $W$. Since $W_{R-1}$ is a basis of $L_{R-1}(\SS)$,  $W_{R}$ is a basis of $L_{R}(\SS)$ and $W_R \setminus W_{R-1}$ consists of irreducible words of length $R$. For Item 3 we define $t_1^{(R)}$ and $t_2^{(R)}$ as follows:

\begin{itemize}

\item For $k \in \{s_1+1,\ldots, s_1 + \ldots +s_{R-1} \}$ we define $t_i^{(R)}(k) =t_i^{(R-1)}(k)$, $i=1,2$. This guarantees Item 3 a,b for all indices $k,h_1,h_2$ less than or equal to $s_1 + \ldots +s_{R-1}$.

\item For $k \in \{s_1 + \ldots +s_{R-1}+1, \ldots, s_1 + \ldots +s_{R-1}+s_R  \}$ we define $t_i^{(R)}(k)$ using the aforementioned correspondence between elements of lengths $R$ and pairs of indices, with $t_1^{(R)}(k)$ being the first and $t_2^{(R)}(k)$ being the second index respectively. Item 3a holds by construction.

Item 3b requires us to check that $t_1^{(R)}(h_1) \le t_1^{(R)}(h_2)$  or $t_2^{(R)}(h_1) \le t_2^{(R)}(h_2)$ for $h_1,h_2$ such that $s_1 < h_1 < h_2 \le s_1 + \ldots +s_R$ with $h_2 \ge s_1 + \ldots +s_{R-1} +1$ (as lower $h_2$ are covered by the induction hypothesis). The proof of this property splits into two cases:

1. If $h_1 \le s_1 + \ldots +s_{R-1}$, then $m_{h_1} < R= m_{h_2}$. By reducing $e_{h_1} = e_{t_1^{(R)}(h_1)} \cdot e_{t_2^{(R)}(h_1)}$ and  $e_{h_2} = e_{t_1^{(R)}(h_2)} \cdot e_{t_2^{(R)}(h_2)}$ to an equation on lengths we get $m_{h_1} = m_{t_1^{(R)}(h_1)} + m_{t_2^{(R)}(h_1)}$ and  $m_{h_2} = m_{t_1^{(R)}(h_2)} + m_{t_2^{(R)}(h_2)}$. Since $x_1 + y_1 < x_2 + y_2$ means that $x_1 <x_2$ or $y_1 < y_2$, we can infer that $m_{t_1^{(R)}(h_1)} < m_{t_1^{(R)}(h_2)}$ or $m_{t_2^{(R)}(h_1)} < m_{t_2^{(R)}(h_2)}$. This provides the inequalities on indices due to $M$ being non-decreasing.

2. If $h_1 \ge s_1 + \ldots +s_{R-1} +1$ the inequalities hold due to the lexicographic sorting of the elements corresponding to pairs.

\end{itemize}

\end{proof}

\begin{proposition}\label{prop_goodbase}
Consider an algebra $\A$ of dimension $n \ge 2$ over a field $\F$ and its generating set $\SS$ such that $l(\SS) \ge 2$. Let $M=(m_0,\ldots,m_{n-1})$ be the characteristic sequence of $\SS$. There exists a basis $\{e_0,\ldots,e_{n-1}\}$ of the algebra $\A$ and functions $t_1,t_2$ from the set $\{k_1 +1,\ldots,n-1\}$ to $\{1,\ldots,n-1\}$, here ${k_1}$ is defined by $m_{k_1}=1$, $m_{k_1+1}>1$, which satisfy the following properties:

\begin{enumerate}
	\item  for all $  i \in \{0,\ldots,n-1 \}$ the element $e_i$ is a word in $\SS$ of length $m_i$.
	\item For $t_1$ and $t_2$ it holds that

		\begin{itemize} \item[a.] For $k$ such that $k_1 <k <n$ the equality  $m_{t_1(k)}+m_{t_2(k)}=m_k$ and inequalities $t_1(k), t_2(k) <k$ hold. \item[b.] For all $h_1,h_2$ such that $k_1<h_1<h_2<n$   at least one of the following two inequalities holds: $t_1(h_1)<t_1(h_2)$ or $t_2(h_1)<t_2(h_2)$. \end{itemize}

	\item  For  any $k$,   $k_1 <k <n$, it holds that $e_{t_1(k)} e_{t_2(k)} = e_k$.
	\item $ \L_1(\SS) = \langle \{ e_0, \ldots, e_{k_1} \} \rangle$.
\end{enumerate}

Particularly, this means that  $\{e_0,\ldots,e_{n-1}\}$ corresponds to the characteristic sequence $M$ equipped with the functions $t_1,t_2$.
\end{proposition}

\begin{proof}

Consider an $\SS$-graded basis $W$, constructed in Lemma \ref{lem_subbase} and numbered in Lemma \ref{lem_meekbase} with $t_1 = t_1^{(l(\SS))}$ and $t_2 = t_2^{(l(\SS))}$ provided by Lemma \ref{lem_meekbase}. We will demonstrate that it satisfies Items 1-4.

\begin{itemize}

\item  $W_1$, which is a basis of $\L_1(\SS)$, coincides with $\{e_0,\ldots,e_{s_1}\}$. As $e_0 = 1$ and $s_1 = k_1$, this implies Item 4 and Item 1 for $i=0,\ldots,s_1$.

\item For any $  j \in \{s_1+1,\ldots,n-1 \}$ there exists $R$  such that $s_1 + \ldots +s_{R-1} < j \le s_1 + \ldots + s_{R}$. We have $m_j = R$ and by Item 2 of Lemma \ref{lem_meekbase} $e_j$ has length $m_j$, which demonstrates the desired property 1 for all remaining $j$.

\item Items 2 and 3 of Lemma \ref{lem_meekbase} for $r=l(\SS)$ guarantee Items 2 and 3 of the proposition.

\end{itemize}

\end{proof}

\section{Algebras of maximal length}

With the results of the previous section we are ready to characterize algebras of maximal length for a given dimension using their basis.

\begin{proposition}\label{prop_powtwo} 
Consider a unital algebra $\A$ over a field $\F$ of dimension $n > 2$ and maximal possible length. There exists a generating set $\SS$ of $\A$ such that its characteristic sequence is $(0,1,2,4,\ldots,2^{n-2})$.
\end{proposition}
\begin{proof}
By Proposition \ref{prop_max} the maximal possible length of a unital algebra with dimension $n$ is $2^{n-2}$. Consider a generating set $\SS$ of such an algebra which has exactly this length. The last element of its characteristic sequence $M = (m_0,\ldots,m_{n-1})$ is equal to $l(\SS)=2^{n-2}$. By Proposition \ref{pr_2} Item 2 this means that all previous $m_h$ for $h>0$ are equal to $2^{h-1}$. Finally, $m_0$ is always equal to zero.
\end{proof}

\begin{definition}
A basis  $\{e_0, e_1, \ldots, e_{n-1}\}$ of an algebra $\A$ is called {\em long}, if $e_0 =1$ and  $e_i^2 = e_{i+1}$ for $i=1,\ldots, n-2$.
\end{definition}

\begin{proposition}
A unital $\F$-algebra $\A$ of dimension $n>2$ which has a long basis $E = \{e_0, e_1, \ldots, e_{n-1} \}$ such that $e_p e_q \in \langle e_0, \ldots, e_{\max (p,q)} \rangle$ for $p \neq q$ also has length $l(\A)=2^{n-2}$.
\end{proposition}

\begin{proof}

Consider the generating set $\{e_1\}$ of $\A$.

We  demonstrate using induction on $j$ that for each $j \in \{1,\ldots,n-1 \}$ the following two statements hold:

\begin{itemize} 

\item There are no irreducible words in $\A$ of lengths $2^{j-2}+1,\ldots,2^{j-1}-1$ for $j > 1$;

\item There exists a single irreducible word  $w_j\in \A$ of length $2^{j-1}$, which is equal to $e_j$ as an element of $\A$ and to $w^2_{j-1}$ as a word, with $w_1$ being~$e_1$.

\end{itemize}

The base. For $j=1$ the statement holds as $w_1=e_1$ is an irreducible word of length $1$ in $\{e_1\}$ with no other possible words of length 1. For $j=2$ the statement holds as $w_2 = w_1^2 = e_1^2 = e_2$ is an irreducible word of length $2$ in $\{e_1\}$ with no other possible words of length 2 and the set $\{2^{j-2}+1,\ldots,2^{j-1}-1\}$ is empty for $j =2$.

The step. Assume the statement holds for $j=1,\ldots,J-1$ with $3 \le J \le n-1$. 

Assume that $v$ is an irreducible word of length $l \in \{2^{J-2}+1, \ldots, 2^{J-1} \}$. Since $J \ge 3$, $l \ge 2$. This means that $v=v' \cdot v''$, where $v'$ and $v''$ are irreducible words of lesser positive lengths, $l'$ and $l''$. As $l' + l'' = l$, the larger one of them is greater than or equal to $l/2$.

1. $l' \ge l''$. Since $l/2 \le l' < l$, this means that $l'$ belongs to $ \{2^{J-3}+1, \ldots, 2^{J-2}\}$. Since $v'$ is irreducible, by the induction hypothesis this means that $v' =w_{J-1}= e_{J-1}$ and $l' = 2^{J-2}$. 

1.1. If $l'' < 2^{J-2}$, then by the induction hypothesis  $v''$ must be equal to $w_r$ with $r\in\{1,\ldots,J-2\}$. However, this would mean $v=w_{J-1} w_r = e_{J-1} e_r \in \langle e_1, \ldots, e_{J-1} \rangle = \L_{2^{J-2}}(\{e_1\})$, with $l>2^{J-2}$. This contradicts to the irreducibility of $v$. Thus, there are no irreducible words of lengths between $2^{J-2} + 1$ and $2^{J-1} -1$. From this follows $\L_{2^{J-1}-1} ( \{e_1\}) = L_{2^{J-2}} ( \{e_1\})$.

1.2. If $l'' = 2^{J-2}$, then $v'' = w_{J-1}$ as well and $v = w_{J-1}^2=:w_J$. This word is indeed irreducible as $w_J = e_J \not\in \L_{2^{J-1}-1} ( \{e_1\}) = L_{2^{J-2}} ( \{e_1\})  = \langle e_0,\ldots, e_{J-1} \rangle$.

2. A similar argument works in the case $l'' \ge l'$. 

From this statement for $j=n-1$ it follows that there exists an irreducible word in $\{e_1\}$ of length $2^{n-2}$, which means $2^{n-2} \le l(\{e_1\}) \le l(\A)$. However, by Proposition \ref{prop_max}, $l(\A) \le 2^{n-2}$, which allows us to conclude that $l(\A) = 2^{n-2}$.

\end{proof}

 \begin{proposition}
A unital $\F$-algebra $\A$ of dimension $n>2$ which has maximal length $2^ {n-2}$ also has a long basis $E = \{e_0, e_1, \ldots, e_{n-1} \}$ such that $e_p e_q \in \langle e_0, \ldots, e_{\max (p,q)} \rangle$ for $p \neq q$.
\end{proposition}

\begin{proof}

Consider a generating set $\bar{\SS}$ of $\A$ such that $l(\bar{\SS}) = l(\A)$. By Proposition \ref{prop_powtwo}, its characteristic sequence is $(0,1,2,4,\ldots,2^{n-2})$. Since there is only one element in the characteristic sequence equal to 1, we have $\dim L_1 (\bar{\SS}) - \dim L_0 (\bar{\SS}) = 1$. This means that we can find a subset $\SS \subset \bar{\SS}$ such that $|\SS| =1$ and $\langle \SS \cup \{1\} \rangle = \langle \bar{\SS} \cup \{1\} \rangle$. By Lemma \ref{lem_redux} $\SS$  is also a generating set of $\A$ of the same length and with the same characteristic sequence, since $\L_k(\SS) = \L_k(\bar{\SS})$ for all natural $k$.   

Note that there is only one way for $(m_0,\ldots,m_{n-1}) = (0,\ldots,2^{n-2})$ to define functions $t_1,t_2$  from the set $\{2,\ldots,n-1\}$ to $\{1,\ldots,n-1\}$ so they would satisfy

\begin{itemize} \item[a.] For $k$ such that $1 <k <n$ the equality  $m_{t_1(k)}+m_{t_2(k)}=m_k$ and inequalities $t_1(k), t_2(k) <k$ hold. \item[b.] For all $h_1,h_2$ such that $1<h_1<h_2<n$   at least one of the following two inequalities holds: $t_1(h_1)<t_1(h_2)$ or $t_2(h_1)<t_2(h_2)$. \end{itemize}

Namely, we must set $t_1(h)=t_2(h)=h-1$, as otherwise $m_{t_1(h)} + m_{t_2(h)} = 2^{t_1(h) -1} + 2^{t_2(h)-1}$ is strictly less than $m_h = 2^{h-1}$. 

This means that by Proposition \ref{prop_goodbase} applied to the generating set $\SS$ there is a basis $\{e_0, \ldots, e_{n-1}\}$ such that $e_0 =1$, $\SS = \{e_1\}$ (as $e_1$ is a word of length 1 in singleton $\SS$), $e_i$ is a word of length $m_i = 2^{i-1}$ for $i=1,\ldots, n-1$ in $\SS$ and $e_j ^2 = e_{j+1}$ for $j=1,\ldots, n-2$ as $t_1(j+1) = t_2(j+1)=j$. In particular, this basis is long.

Additionally, $e_i$ is an irreducible word. For $i=0,1$ this is evident. To prove this for $i \ge 2$, note that the dimension of $\L_{2^{i-2}}(\SS)$ is equal to the number of elements of the characteristic sequence less than or equal to  $2^{i-2}$, i.e.  $i$, which means that $i$ linearly independent words $e_0, \ldots, e_{i-1}$ which belong to  $\L_{2^{i-2}}(\SS)$ form its basis. Also by Lemma \ref{lem_chseq} there are no irreducible words of lengths $2^{i-2}+1, \ldots, 2^{i-1}-1$ as there are no such elements in the characteristic sequence of $\SS$. This allows us to conclude that $\L_{2^{i-2}}(\SS) = \L_{2^{i-1}-1}(\SS) = \langle e_0,\ldots, e_{i-1} \rangle$, from which follows $e_i \not\in  \L_{2^{i-1}-1}(\SS) $.

To demonstrate that $e_p e_q \in \langle e_0, \ldots, e_{\max (p,q)} \rangle$ for $p \neq q$, assume the opposite. Let there be $p,q$ with $p \neq q$ such that  $e_p e_q = f_r e_r + \ldots +f_0 e_0$, where $f_r \neq 0$ and $r > \max(p,q)$. The word $e_r$ of length $2^{r-1}$ is irreducible. However we can represent $e_r = c e_p e_q + s$, where $c = \frac{1}{f_r}$ and $s \in \langle \{e_0, \ldots, e_{r-1} \} \rangle$. This implies  that $e_r \in \L_{\max (2^{r-2}, 2^{p-1} +2^{q-1}) } (\{e_1\})$. Since  $2^{r-2}< 2^{r-1}$ and $2^{p-1} +2^{q-1}<2^{r-1}$, the word $e_r$ is reducible. This is a contradiction, which means that the initial assumption is incorrect.

\end{proof}

The following is immediate from the two propositions above.

\begin{theorem}
A unital $\F$-algebra $\A$ of dimension $n>2$ has maximal length $2^ {n-2}$ if and only if it has a long basis $E = \{e_0, e_1, \ldots, e_{n-1} \}$ such that $e_p e_q \in \langle e_0, \ldots, e_{\max (p,q)} \rangle$ for $p \neq q$.

\end{theorem}

\begin{corollary}
For every unital $\F$-algebra $\A$ with $\dim(\A)=n>2$ and    $l(\A)=2^ {n-2}$ there exist elements $f^{(j)}_{p,q}$ and $f^{(0)},\ldots,f^{(n-1)}$ in $\F$, $p,q \in \{1,\ldots,n-1\},$ $p\neq q$, $j =0,\ldots, \max(p,q)$ such that $\A$ is isomorphic to an algebra defined by generators $x_0,\ldots,x_{n-1}$ and relations
$$x_0 =1, \ x_i^2 = x_{i+1}, \ i =1,\ldots, n-2,$$ $$ x_p x_q = \sum\limits_{j=0}^{\max (p,q)}  f^{(j)}_{p,q}  x_j , \qquad x_{n-1}^2 = \sum\limits_{i=0}^{n-1} f^{(i)} x_i.$$
\end{corollary}

\begin{proof}
By the theorem above, $\A$  has a long basis $E = \{e_0, e_1, \ldots, e_{n-1} \}$ such that $e_p e_q \in \langle e_0, \ldots, e_{\max (p,q)} \rangle$ for $p \neq q$. Using this basis as well as the fact that $e_{n-1}^2 \in \A$ we can find the elements  $f^{(j)}_{p,q}$ and $f^{(0)},\ldots,f^{(n-1)}\in \F$ such  that

$$e_p e_q = \sum\limits_{j=0}^{\max (p,q)}  f^{(j)}_{p,q} e_j , \qquad e_{n-1}^2 = \sum\limits_{i=0}^{n-1} f^{(i)} e_i.$$

This  allows to construct the desired isomorphism simply by mapping $e_i$ onto~$x_i$.
\end{proof}

\end{document}